\documentclass[a4paper,11pt]{amsart}
\usepackage{amssymb,amsmath,amsthm}
\usepackage{amsfonts}
\usepackage{latexsym}
\usepackage{graphicx}
\usepackage{color}
\usepackage{layout}
\usepackage[english]{babel}

\usepackage[active]{srcltx} 

\newcommand{\Z}{\mathbb Z}

\newcommand{\comm}[1]{\marginpar{\tiny #1}}

\def\ds{\displaystyle}

\def\C{\mathbb C}
\def\R{\mathbb R}
\def\Q{\mathbb Q}
\def\Z{\mathbb Z}

\def\F{\mathbb F}

\DeclareMathOperator{\Li}{Li}

\newcommand{\leg}[2]{\left(\frac{#1}{#2}\right)}

\theoremstyle{theorem}
\newtheorem{thm}{Theorem}[section]
\newtheorem{lem}[thm]{Lemma}

\theoremstyle{remark}

\begin{document}

\title[Congruences for central binomial sums and finite polylogarithms]{Congruences for central binomial sums \\ and finite polylogarithms}

\author{Sandro Mattarei}
\email{\tt mattarei@science.unitn.it}
\address{Dipartimento di Matematica,
Universit\`a di Trento, via Sommarive 14,
38123 Trento, Italy}

\author{Roberto Tauraso}
\email{tauraso@mat.uniroma2.it}
\address{Dipartimento di Matematica, 
Universit\`a di Roma ``Tor Vergata'', 
via della Ricerca Scientifica, 
00133 Roma, Italy}

\date{\today}
\maketitle
\begin{abstract}
We prove congruences, modulo a power of a prime $p$, for certain finite
sums involving central binomial coefficients $\binom{2k}{k}$,
partly motivated by analogies with the well-known power series
for $(\arcsin z)^2$ and $(\arcsin z)^4$.
The right-hand sides of those congruences involve
values of the finite polylogarithms
$\pounds_d(x)=\sum_{k=1}^{p-1} x^k/k^d$.
Exploiting the available functional equations for the latter
we compute those values, modulo the required powers of $p$,
in terms of familiar quantities such as Fermat quotients and Bernoulli numbers.
\end{abstract}

\section{Introduction}\label{sec:intro}

A well-known power series expansion of a familiar function where central binomial coefficients appear in the denominators is
\[
2\bigl(\arcsin(z/2)\bigr)^2
=\sum_{k=1}^{\infty} \frac{1}{k^2\binom{2k}{k}}\,z^{2k},
\]
which yields
$\sum_{k=1}^{\infty} k^{-2}\binom{2k}{k}^{-1}=\pi^2/18=\zeta(2)/3$
upon setting $z=1$.
By appropriate successive applications of differentiation one can derive explicit closed-form expressions
for the power series
$\sum_{k=1}^{\infty} k^{-d}\binom{2k}{k}^{-1}t^k$,
for any integer $d\le 2$,
and corresponding specializations to interesting values of $t$.
For $d>2$ the sum of that power series appears not to be expressible in terms of the simpler transcendental functions,
but explicit evaluations at special values of $t$ are still possible, such as
$\sum_{k=1}^{\infty} (-1)^k k^{-3}\binom{2k}{k}^{-1} =-2\,\zeta(3)/5$ and
$\sum_{k=1}^{\infty} k^{-4}\binom{2k}{k}^{-1}=17\,\zeta(4)/36$.
The former played a role in Apery's celebrated proof of irrationality of $\zeta(3)$,
see van der Poorten's account of Apery's proof~\cite{vdP:79} for a discussion of both formulas and relevant references.
Evaluation of the series for $d$ up to $8$ when $t=1$, and $d$ up to $9$ when $t=-1$, were found in~\cite{BCK:01},
exploiting special values of polylogarithms whose availability depends on {\em polylogarithm ladders}~\cite{Lew:91},
and hence, ultimately, on functional equations satisfied by the classical polylogarithms
$\Li_m(z)=\sum_{k=1}^{\infty}z^k/k^m$.

In a different direction, the power series expansions for $(\arcsin z)^m$
were determined in~\cite{BC:07}, extending on the known results for $m=1,\ldots,4$
(see~\cite[pp.~262--263]{Be:98}, for example).
Besides certain types of {\em multiple harmonic sums,} the coefficients involve
a central binomial coefficient in the numerator for $m$ odd, and in the denominator
for $n$ even.
Of special interest for us is the case $m=4$, which reads
\begin{equation}\label{eq:asH}
\frac{2}{3}\bigl(\arcsin(z/2)\bigr)^4
=
\sum_{k=1}^{\infty} \frac{H_{k-1}(2)}{k^2\binom{2k}{k}}\,z^{2k},
\end{equation}
where
$H_{k-1}(2)=\sum_{r=1}^{k-1}1/r^2$.
Again, differentiation produces similar closed-form expressions
for the sums of analogous power series with $k$ or $1$ in place of the factor $k^2$ at the denominator (as in~\cite{CZ:09}, for example).

Finite modular versions of familiar power series play a role in various parts of algebra and number theory,
where a power series is truncated at an appropriate place so that the remaining coefficients are $p$-integral,
thus obtaining a polynomial which can be evaluated modulo $p$.
Some of the functional properties of the sum of the infinite series may be preserved in that polynomial.
A distinguished algebraic example is the crucial role of the truncated exponential series $\sum_{k=0}^{p-1}x^k/k!$
in the theory of modular Lie algebras, as a tool for {\em toral switching}~\cite[Chapter~1]{Str:04}:
little is preserved of the functional equation $\exp{(x+y)}=\exp(y)\exp(y)$,
but just enough to make the algebraic application work, see~\cite{Mat:05} for an extension
of this point of view.
As an example from number theory we mention the use of the partial sum
$\sum_{k=1}^{p-1}x^k/k$ of the logarithmic series $-\log(1-x)$ made in~\cite{HB:02}:
there a polynomial argument about the partial sum is strongly motivated by transcendence arguments for the logarithmic function.
Generally speaking, when an infinite power series with rational coefficients admits an explicit summation formula
it is natural to seek for finite modular analogues, that is, for congruences modulo $p$ or a power of $p$
for an appropriate truncated version of the series, and see how those resemble the original function.

In this note we consider the sums of the first $p-1$ terms of some of the series mentioned earlier,
where $p$ is a prime, and evaluate them modulo certain powers of $p$.
Specifically, we obtain congruences for the polynomials
\begin{equation}\label{eq:MC}
p\sum_{k=1}^{p-1} \frac{t^k}{k^d\binom{2k}{k}}\pmod{p^3},
\quad\mbox{and}\quad
p\sum_{k=1}^{p-1} \frac{H_{k-1}(2)}{k^d\binom{2k}{k}}\,t^k \pmod{p},
\end{equation}
where $p$ is a prime and $d=0,1,2$
(and possibly $d=3,4$ as well, as we discuss at the end of this Introduction),
which we then specialize to particular values of $t$.
(Multiplication by $p$ is needed to make the resulting coefficients $p$-integral.)
Special cases of the second type of sum above were considered by Z.~W.~Sun in~\cite{ZWS:09}
together with related sums, for certain values of $t$,
and with attention to a comparison with the corresponding infinite sums.
As we explain in our Section~\ref{sec:ZWS}, our results include a few congruences first obtained in~\cite{ZWS:09}.
However, we produce many new ones in a systematic way, and provide a framework
to possibly obtain more.
As a test of the validity of this approach we prove several conjectures formulated by Z.~W.~Sun in~\cite{ZWS:10}.

A crucial observation is that, in analogy with the corresponding infinite sums,
explicit evaluation of the sums in Equation~\eqref{eq:MC} for specific values of $t$ depends on the availability
of special values of the {\em finite polylogaritms,} defined as
\[
\pounds_d(x)=\sum_{k=1}^{p-1} \frac{x^k}{k^d},
\]
where $d$ is a positive integer.
In turn, the possibility of computing those modulo small powers of $p$
is due to the existence of several known functional equations (in the shape of congruences)
satisfied by the finite polylogarithms, which we collect in Section~\ref{sec:polylog}.

It is fair to assume that much of this material on finite polylogarithms was known to Mirimanoff at the beginning
of the twentieth century.
In fact, two special functional equations (modulo $p$) relating $\pounds_1(x)^2$ and $\pounds_1(x)^3$
to values of $\pounds_2$ and $\pounds_3$, which were rediscovered in~\cite{Gr:04} and~\cite{DS:06},
were explicitly mentioned by Mirimanoff in~\cite[p.~61]{Mir:04}.
Because Mirimanoff omitted the proofs, and the proofs by algebraic manipulations given in~\cite{Gr:04} and~\cite{DS:06}
tend to hide how such equations might be discovered in the first place,
we devote Section~\ref{sec:polylog-proofs} to presenting our own proofs of those polynomial congruences.
The crux of our argument is that while the initial coefficients of $\pounds_1(x)^2$ and $\pounds_1(x)^3$
are easy to obtain as in the characteristic-zero case, invariance under a certain (rather illustrious) symmetry group of order six allows one to
recover all of the remaining coefficients.

For certain special values of $x$ the available functional equations for finite polylogarithms
taken together provide enough information
to determine $\pounds_d(x)$ modulo $p$, for $d=1,2,3$.
We present these evaluations in Section~\ref{sec:special_values}.

In Sections~\ref{sec:identities} and~\ref{sec:congruences}
we establish the necessary connection between the sums in Equation~\ref{eq:MC} and values of finite polylogarithms.
This does require some work, which we split into two parts and outline here.
The first part, in Section~\ref{sec:identities},
is to produce polynomial identities (that is, in characteristic zero)
which express our sums in Equation~\eqref{eq:MC} as more tractable sums
involving {\em Dickson polynomials}.
Because Dickson polynomials satisfy second-order linear recurrence relations,
certain sums involving them can be expressed in terms of finite polylogarithms.
However, bringing the sums of Section~\ref{sec:identities} to the required form
requires passing from polynomial identities to polynomial congruences, which we do in Section~\ref{sec:congruences}.

We devote Section~\ref{sec:numerators} to simpler-looking polynomials obtained
from those in Equation~\eqref{eq:MC} by switching the central
binomial coefficients from the denominators to the numerators.
Congruences for them cannot, generally speaking, be inferred from the corresponding ones
for the polynomials in Equation~\eqref{eq:MC}, but they can be
obtained by similar methods, and also involve values of the finite polylogarithms.

Our final Section~\ref{sec:ZWS} brings together the two main streams of
this paper, namely, the finite polylogarithms studied in
Sections~\ref{sec:polylog}, \ref{sec:polylog-proofs}
and~\ref{sec:special_values}, and the polynomial congruences
developed through Sections~\ref{sec:identities},
\ref{sec:congruences} and~\ref{sec:numerators}.
The polynomial congruences for the sums in Equation~\eqref{eq:MC}
and their analogues with the central binomial coefficients in the
numerators can be evaluated at the special values of $t$
for which we have computed the relevant finite polylogarithmic
values in Section~\ref{sec:special_values}.
Many numerical congruences can be obtained in this way, and we restrain ourselves
to display a selection of the most interesting ones, which include several conjectured by Z.~W.~Sun.

A few words are appropriate to comment on our restriction $d\le 2$ for the polynomials in Equation~\eqref{eq:MC}.
In principle our polynomial identities in Section~\ref{sec:identities} can be extended to higher values of $d$,
each case following from the previous one by appropriate integration.
In fact, the third identity in our Theorem~\ref{thm:recurrences} is for $d=3$, and then leads to the congruence in~Theorem~\ref{thm:congruences3}.
In Section~\ref{sec:ZWS} we apply the corresponding polynomial identity with $d=4$ without actually stating it;
one can find it quoted in~\cite{PP:11}.
However, it does not appear feasible to obtain pleasant numerical congruences from those polynomial identities
for higher values of $d$.

\section{General congruences for $\pounds_d(x)$}\label{sec:polylog}

In this section we collect some functional equations modulo a prime $p$
and other relations satisfied by the finite polylogarithms, especially $\pounds_1$, $\pounds_2$ and $\pounds_3$,
which we will use in the rest of the paper.
Some of them are related to functional equations satisfied by the classical polylogarithms (see~\cite{Lew:81});
a procedure for deducing them from the latter is described in~\cite{EG:02}.
The following most basic identities actually hold for all finite polylogarithms $\pounds_d$:
\begin{itemize}
\item
the inversion relation~\cite[Proposition~5.7(1)]{EG:02}, and its extension modulo $p^2$~\cite[Lemma~4.3]{ZHS:08b},
\begin{align}\label{C1}
&\pounds_d(x) \equiv (-1)^d x^p \pounds_d(1/x) \pmod{p},
\\
\label{C1b}
&\pounds_d(x) \equiv (-1)^d x^p \pounds_d(1/x) -dp\pounds_{d+1}(x) \pmod{p^2};
\end{align}
\item
the distribution relation~\cite[Proposition~5.7(2)]{EG:02},
\begin{equation}\label{C6}
\pounds_d(x^m)\equiv
m^{d-1}\sum_{k=0}^{m-1}\biggl(
\sum_{j=0}^{m-1} (\omega_m^kx)^{pj}\biggr)\pounds_d(\omega_m^k x) \pmod{p},\quad\mbox{where $\omega_m=e^{2\pi i/m}$};
\end{equation}
of course this congruence takes place in the ring of integers of the cyclotomic field $\Q(\omega)$.
\end{itemize}

Next, there are several relations which are specific to $\pounds_1$, $\pounds_2$ and $\pounds_3$.
Some of them involve the quantities
\[
q_p(x)=\frac{x^{p-1}-1}{p}\quad\mbox{and}\quad
Q_p(x)=\frac{x^{p}+(1-x)^p-1}{p}=xq_p(x)+(1-x)q_p(1-x),
\]
and some require $p>3$, which we assume from now on for simplicity.
They are as follows:
\begin{itemize}
\item
the 3-term relation for $\pounds_2$~\cite[Proposition~5.11]{EG:02}, rediscovered in~\cite[Equation~(5)]{Gr:04},
\begin{equation}\label{C3}
\pounds_2(x) \equiv \pounds_2(1-x)+x^p\pounds_2(1-1/x) \pmod{p};
\end{equation}
\item
a congruence noted by Granville~\cite[Equation~(6)]{Gr:04},
\begin{equation}\label{C2}
Q_p(x) \equiv -\pounds_1(1-x)-p\pounds_2(x) \pmod{p^2};
\end{equation}
\item
another congruence rediscovered by Granville~\cite[Equation~(5)]{Gr:04}, but see~\cite[p.~61]{Mir:04},
\begin{equation}\label{C4}
\frac{1}{2}Q_p^2(x) \equiv -x^p\pounds_2(x)-(1-x^p)\pounds_2(1-x) \pmod{p};
\end{equation}
\item
a congruence rediscovered by Dilcher and Skula~\cite[Theorem~2]{DS:06}, but see~\cite[p.~61]{Mir:04},
\begin{equation}\label{C5}
\begin{split}
\frac{1}{6}Q_p^3(x) &\equiv -x^p\pounds_3(x)-(1-x^p)\pounds_3(1-x)-x^{2p}(1-x^p)\pounds_3(1-1/x) \\
&\quad-\frac{2}{3}x^p(1-x^p)\pounds_3(-1) \pmod{p}.
\end{split}
\end{equation}
\end{itemize}

We will also need a special case of the following congruence, obtained by the authors in~\cite[Lemma~3.2]{MT:10}: for $p>d+1$
\begin{equation}\label{polylogMT}
\sum_{0<k_1<k_2<\cdots<k_d<p}\frac{x^{k_d}}{k_1 k_2\cdots k_d}
\equiv (-1)^{d-1}\pounds_d(1-x)\pmod{p}.
\end{equation}

We mention for completeness that the easy congruence~\eqref{C1b}
can be extended as follows modulo arbitrary powers of $p$,
\begin{equation}\label{C1c}
(-1)^d x^p \pounds_d(1/x)
=\sum_{m=0}^{\infty}\binom{d+m-1}{m}p^m\pounds_{d+m}(x),
\end{equation}
to be interpreted in the power series ring $\Z_p[[x]]$.

\begin{proof}[Proof of Equation~\eqref{C1c}]
We have
\begin{align*}
(-1)^d x^p \pounds_d(1/x)
&=(-1)^d\sum_{k=1}^{p-1}\frac{x^{p-k}}{k^d}
=(-1)^d\sum_{k=1}^{p-1}\frac{x^k}{(p-k)^d}
\\&=
\sum_{k=1}^{p-1}\frac{x^k}{k^d}\left(1-\frac{p}{k}\right)^{-d}
\\&=
\sum_{k=1}^{p-1}\frac{x^k}{k^d}\sum_{m=0}^{\infty}\binom{d+m-1}{d}(p/k)^m
\\&=
\sum_{m=0}^{\infty}\binom{d+m-1}{m}p^m\pounds_{d+m}(x),
\end{align*}
as desired.
\end{proof}

\section{New proofs of congruences~\eqref{C3}, \eqref{C4} and~\eqref{C5}}\label{sec:polylog-proofs}

The proofs of Equations~\eqref{C3}, \eqref{C2} and~\eqref{C4} given in~\cite{Gr:04},
and that of Equation~\eqref{C5} in~\cite{DS:06}, were obtained by algebraic manipulations after differentiation of both sides.
An undesirable feature of such proofs is that one is required to guess the desired congruence in the first place.
We present proofs of Equations~\eqref{C3}, \eqref{C4} and~\eqref{C5}
which do not suffer from this imperfection.

Because of the congruence
\begin{equation}\label{eq:Q}
Q_p(x)\equiv -\pounds_1(x)\pmod{p},
\end{equation}
which plainly follows from the definition of $Q(x)$ by expanding $(1-x)^p$ and using the fact that
$\binom{p}{k}=\frac{p}{k}\binom{p-1}{k-1}\equiv(-1)^{k-1}/p\pmod{p^2}$ for $0<k<p$,
Equations~\eqref{C4} and~\eqref{C5} are seen to be equivalent to the second and third of the following set of three congruences:
\begin{align}
\label{eq:L^1}
\pounds_1(x)
&\equiv
\pounds_1(1-x)\pmod{p},
\\\label{eq:L^2}
\pounds_1(x)^2/2
&\equiv -x^p\pounds_2(x)-(1-x^p)\pounds_2(1-x) \pmod{p},
\\\label{eq:L^3}
\pounds_1(x)^3/6
&\equiv
x^p\pounds_3(x)+(1-x^p)\pounds_3(1-x)+x^{2p}(1-x^p)\pounds_3(1-1/x)
\\&\notag
\quad
+(2/3)x^p(1-x^p)\pounds_3(-1) \pmod{p};
\end{align}
the second congruence clearly requires $p>2$, and the third one $p>3$.
The first of these three congruences follows from Equation~\eqref{eq:Q} and the obvious invariance of $Q_p(x)$
under the substitution $x\mapsto 1-x$.
The remaining two were already known to Mirimanoff~\cite[p.~61]{Mir:04}, as we pointed out in the Introduction.
Note that the terms of degree less than $p$ in the right-hand sides of the three congruences
are given by $\pounds_1(1-x)$, $-\pounds_2(1-x)$ and $\pounds_3(1-x)$.
It is easy to see that these terms match the corresponding terms in the left-hand sides.
In fact, this statement appropriately extends to powers $\pounds_1(x)^d$ for arbitrary $d$, as we show in
Lemma~\ref{lem:pounds-powers} below, including some extra terms as well.
It follows that congruences~\eqref{eq:L^2} and~\eqref{eq:L^3}
are verified up to and including the term of degree $p$.
Then we will recover all the remaining terms in the right-hand sides of Equations~\eqref{eq:L^2} and~\eqref{eq:L^3},
and thus complete their proofs,
by invariance under a group of transformations of order six, generated by the symmetry expressed by Equation~\eqref{eq:L^1}
together with the other obvious symmetry
$\pounds_1(x) \equiv -x^p \pounds_1(1/x) \pmod{p}$,
which is a special case of Equation~\eqref{C1}.
In case of Equation~\eqref{eq:L^2}, where only about half the coefficients need to be recovered,
the argument yields a proof of Equation~\eqref{C3} as a by-product.
The group of transformations of order six has a long history, being omnipresent in the investigations
on Fermat's last theorem, see~\cite{Rib:79}, and it is a fair guess that Mirimanoff's own proofs of
congruences~\eqref{eq:L^2} and~\eqref{eq:L^3} might have had much in common with ours.

Because $\pounds_1(x)\equiv -\log(1-x)\pmod{x^p}$,
we start the ball rolling by studying the coefficients in the powers of the ordinary logarithmic series.

\begin{lem}\label{lem:log-powers}
For any nonnegative integers $d,k$, the coefficient of $x^k$ in the power series
\[
\log^d(1+x)/d!\in\Q[[x]]
\]
equals the coefficient of $y^d$ in the polynomial
\[
\binom{y}{k}=y(y-1)\cdots(y-k+1)/k!\in\Q[y].
\]
\end{lem}

\begin{proof}
The identity
\[
\exp(y\log(1+x))=(1+x)^y
\]
yields
\[
\sum_{d=0}^{\infty}(y\log(1+x))^d/d!=
\sum_{k=0}^{\infty}\binom{y}{k}x^k,
\]
with both series converging for $(x,y)$ in a suitable neighbourhood of the origin in $\R^2$ (or $\C^2$).
Hence the latter identity holds in the ring of formal power series $\Q[[x,y]]$,
and the conclusion follows.
\end{proof}

Our usage of polynomial congruences with respect to a double modulus $(x^m,p)$ will be to give
precedence to the modulus $x^m$ over the modulus $p$, in the sense that
we interpret them as congruences modulo $p$ after all terms of degree $m$ or higher have been discarded
(regardless of their coefficients).

\begin{lem}\label{lem:pounds-powers}
For $1<d<p-1$ we have
\[
\pounds_1(x)^d/d!\equiv(-1)^{d-1}\pounds_d(1-x)
+(-1)^d\frac{B_{p-d}}{d}x^p
\pmod{(x^{p+1},p)},
\]
where
$B_{p-d}$ denotes a Bernoulli number.
\end{lem}

\begin{proof}
The terms of degree less than $p$ in the polynomial $\pounds_1(x)$
coincide with the corresponding terms of the power series $-\log(1-x)$.
Because there is no constant term,
Lemma~\ref{lem:log-powers} implies that the coefficient of $x^k$ in the polynomial $\pounds_1(x)^d/d!$,
for $k<p+d-1$, equals $(-1)^{d+k}$ times the coefficient of $y^d$ in the polynomial
$\binom{y}{k}\in\Q[y]$.
In particular, this holds for $k\le p$, which is all we need here.

As for the the first term at the right-hand side of the congruence, we have
\[
\pounds_d(1-x)=\sum_{r=1}^{p-1}\frac{(1-x)^r}{r^d}
=
\sum_{r=1}^{p-1}\frac{1}{r^d}\sum_{k=0}^{p-1}\binom{r}{k}(-x)^k
=
\sum_{k=0}^{p-1}\Biggl(\sum_{r=1}^{p-1}\frac{1}{r^d}\binom{r}{k}\Biggr)(-x)^k.
\]
Because $\sum_{r=1}^{p-1}r^h\equiv -1\pmod{p}$ if $p-1\mid h$, and $\equiv 0\pmod{p}$ otherwise,
after expanding $\binom{r}{k}$ as a polynomial in $r$ we see that the sum
\[
\sum_{r=1}^{p-1}\frac{1}{r^d}\binom{r}{k}
=
\sum_{r=1}^{p-1}\frac{1}{r^d}(a_kr^k+\cdots+a_1r+a_0)
\]
is congruent, modulo $p$, to the opposite of the coefficient of $y^d$ in the polynomial
$\binom{y}{k}\in\Q[y]$, for $k<p$.
This proves that the stated congruence holds modulo $(x^p,p)$.

We deal with the coefficient of $x^p$ noting that
\[
\binom{y}{p}=\frac{y}{p}\prod_{j=1}^{p-1}\left(1-\frac{y}{j}\right)
=\frac{y}{p}\sum_{r=0}^{p-1} h_{r}(-y)^r,
\]
where
\[
h_r=\sum_{0<j_1<j_2<\cdots<j_{r}<p}\frac{1}{j_1 j_2\cdots j_r}.
\]
According to~\cite{ZC:07}, for $1\leq r\leq p-3$ we have
\[
h_r\equiv \frac{(-1)^{r-1}}{r+1}\,pB_{p-r-1}\pmod{p^2},
\]
which completes the proof.
\end{proof}

The congruence in Lemma~\ref{lem:pounds-powers} traces back to Mirimanoff.
With little extra effort the above proof extends it to a congruence
modulo $(x^{p+d-1},p)$ involving Stirling numbers of the first kind besides Bernoulli numbers.

According to Lemma~\ref{lem:pounds-powers} the term of degree $p$ in $\pounds_d(x)^d$ vanishes modulo $p$ when $d$ is even.
An alternative route to this conclusion is noting that the polynomial $\binom{y}{p}-(y^p-y)/p$ has $p$-integral coefficients
and that its reduction modulo $p$ is an odd polynomial in $\F_p[y]$,
which is easy to check by evaluating it on $y=0,1,\ldots, p-1$.

Lemma~\ref{lem:pounds-powers} tells us that congruence~\eqref{eq:L^2} is correct
as far as we look only at the terms of degree up to $p$.
To complete the proof we now use the available symmetries.

\begin{proof}[Proof of Equations~\eqref{C3} and~\eqref{eq:L^2}]
According to Equation~\eqref{C1} we have
$\pounds_1(x)^2\equiv x^{2p}\pounds_1(1/x)^2\pmod{p}$.
This means that the coefficients of $x^k$ and $x^{2p-k}$ in $\pounds_1(x)^2$ are equivalent modulo $p$, for all $k$.
But the values modulo $p$ of the lower half of the coefficients are given in Lemma~\ref{lem:pounds-powers}, namely,
$\pounds_1(x)^2/2\equiv-\pounds_2(1-x)\pmod{(x^{p+1},p)}$.
Hence this determines the upper half of the coefficients as well, and so we have
\begin{equation}\label{eq:G_variant}
\pounds_1(x)^2/2\equiv -\pounds_2(1-x)-x^{2p}\pounds_2(1-1/x) \pmod{p}.
\end{equation}
Because the left-hand side is invariant, modulo $p$, under the substitution $x\mapsto 1-x$, so
must be the right-hand side, and hence
\[
-\pounds_2(1-x)-x^{2p}\pounds_2(1-1/x)
\equiv
-\pounds_2(x)-(1-x)^{2p}\pounds_2(x/(x-1))
\pmod{p}.
\]
Using $\pounds_2(y)=y^p\pounds_2(1/y)$ with $y=x/(x-1)$, and rearranging terms, we obtain Equation~\eqref{C3}.
Substituting it appropriately in Equation~\eqref{eq:G_variant} yields  Equation~\eqref{eq:L^2}.
\end{proof}

We will follow a similar approach to prove Equation~\eqref{eq:L^3}.
However, in this case Lemma~\ref{lem:pounds-powers} provides us with just about one third of the values modulo $p$
of the coefficients of $\pounds_1(x)^3$, and so we need a more careful application of symmetries to recover
the rest of the coefficients.
For this reason we take some time to discuss the group of symmetries in some detail,
and the polynomials which it leaves invariant.

If $F$ is any field, the involutive transformations $R:z\mapsto 1/z$ and $S:z\mapsto 1-z$
of the projective line $F\cup\{\infty\}$ generate a group
\begin{equation}\label{eq:G}
G=\left\{1,R,S,RS,SR,RSR\right\}
\end{equation}
of order six, which is isomorphic to the symmetric group on three objects (with $1$ denoting the identity map).
Thus, writing composition of maps from right to left, the group also contains the two elements
$RS:z\mapsto 1/(1-z)$
and
$SR:z\mapsto 1-1/z$,
which have order three and are inverse of each other,
and a third involution
$RSR=SRS:z\mapsto z/(z-1)$.
As is well known, $R^2=1$, $S^2=1$ and $RSR=SRS$ are a set of defining relations for $G$ as a group generated by $R$ and $S$.

The group $G$ plays a crucial role in virtually all of this paper.
By considering the fixed points of the various elements of $G$ in the action it is easy to see that
all orbits of $G$ on $F\cup\{\infty\}$ have length six, with the only exceptions of the orbits
$\{1,0,\infty\}$ and $\{-1,2,1/2\}$ of length three
(but they coincide if $F$ has characteristic two, and the latter orbit has length one
if $F$ has characteristic three) and,
possibly, an orbit $\{\omega_6,\omega_6^{-1}\}$ of length two (or one if $F$ has characteristic three).
This last orbit exists when $F$ contains a root $\omega_6$
of the polynomial $x^2-x+1$ (which for the finite field $F=\F_q$ is the case if and only if $q\equiv 0,1\pmod{3}$).

This action of $G$ on the projective line $F\cup\{\infty\}$ naturally induces an action
on its function field as an algebraic variety.
A formal treatment would require dealing with homogeneous polynomials and then rational functions in two indeterminates
$x_0$ and $x_1$, but to avoid losing sight of the main argument we prefer to use the affine parameter $x=x_1/x_0$ for the projective line,
at the expense of adding some {\em ad-hoc} terminology concerning the point $\infty$.
(A few comments on the more formal point of view will be added in parentheses for the more algebraically-inclined reader.)

We define a linear representation of $G$ on $F[x]_{\le m}$, the space of polynomials of degree not exceeding $m$, by setting
\[
(Rf)(x):=(-x)^mf(1/x),
\qquad\text{and}\qquad
(Sf)(x):=f(1-x),
\]
for $f\in\F_p[x]_{\le m}$.
That this stipulation really defines a representation of $G$ can be verified by checking that $R(Rf)=f$, $S(Sf)=f$, and
$R(S(Rf)))=S(R(Sf))$.
One finds
\begin{align*}
(RSf)(x)&=(-x)^mf(1-1/x),
\\
(SRf)(x)&=(x-1)^mf\bigl(1/(1-x)\bigr),
\\
(RSRf)(x)&=(x-1)^mf\bigl(x/(x-1)\bigr)=(SRSf)(x).
\end{align*}
(In terms of homogeneous coordinates $(x_0,x_1)$ with $x=x_1/x_0$,
we would obtain this representation of $G$ on $F[x_0,x_1]$ by letting $G$ act on a linear form $f(x_0,x_1)=x_0f(1,x_1/x_0)$ by
$(Rf)(x_0,x_1):=f(-x_1,-x_0)$ and
$(Sf)(x_0,x_1):=f(x_0,x_0-x_1)$.)

Given a polynomial $f\in F[x]$, we may assign to it a {\em formal degree} $m$, any integer no less than the ordinary degree $\deg f$,
to indicate that we are viewing $f$ as an element of $F[x]_{\le m}$ and elements of the group $G$ should act on it as described above.
(Thus, the same polynomial can be assigned different formal degrees.)
Then the action above is compatible with polynomial multiplication, in the sense that
if $f_1$ and $f_2$ are polynomials of formal degrees $m_1$ and $m_2$, and we assign formal degree $m_1+m_2$ to their product $f_1f_2$, then
$T(f_1f_2)=(Tf_1)(Tf_2)$ for any $T\in G$.
(This artifice makes up for not using homogeneous coordinates, and a polynomial of formal degree $m$ really corresponds to a polynomial
function of degree $m$ on the projective line.)
If we agree to say that a polynomial $f$ of formal degree $m$ has the point $\infty$ as a zero with multiplicity $m-\deg f$,
then the sum of the multiplicities of the roots of $f$, including that of $\infty$,
does not exceed its formal degree $m$, unless $f$ is the zero polynomial.

If the field $F$ has characteristic greater than $3$, as we assume from now on, it is a basic fact that the map
$f\mapsto(1/|G|)\sum_{T\in G}Tf$ projects $F[x]_{\le m}$,
the space of polynomials of formal degree $m$,
onto its subspace of $G$-invariant polynomials.
(This is the projection operator used in the standard proof of Maschke's theorem in the basic
representation theory of finite groups, for example.)
Thus, any $G$-invariant polynomial of formal degree $m$ can be expressed as
$f+Rf+Sf+RSf+SRf+RSRf$
for some $f\in F[x]_{\le m}$.
More conveniently for us,
any $G$-invariant polynomial of formal degree $m$ can be expressed as
\begin{equation}\label{eq:G-invariant}
(f+Sf+RSf)(x)=f(x)+f(1-x)+(-x)^m f(1-1/x)
\end{equation}
for some $f\in F[x]_{\le m}$ with the extra property that $Rf=f$.

We are now getting close to a proof of congruence~\eqref{eq:L^3}.
The fact that most orbits of $G$ on $F\cup\{\infty\}$ have length six
implies that, roughly speaking, all the coefficients of a $G$-invariant polynomial $f$ of formal degree $m$
can be recovered from knowledge of only about $m/6$ coefficients, if carefully selected.
Of course we will need to specify a larger number of coefficients if our selection encodes redundant information.
The following lemma shows that the lower third of the coefficient list is a sufficiently large selection to this purpose.

\begin{lem}\label{lem:unique}
Let $f$ be a $G$-invariant polynomial of formal degree $m$.
If $f$ has no terms of degree $\le m/3$, then $f$ is the zero polynomial.
\end{lem}

\begin{proof}
By hypothesis $0$ is a root of $f$ with multiplicity strictly higher than $m/3$.
Recall that the $G$-orbit of $0$ is $\{0,1,\infty\}$.
Invariance under $G$ implies that $1$ and $\infty$ are also roots of $f$, each with multiplicity exceeding $m/3$.
It follows that $f$ is the zero polynomial.
\end{proof}

\begin{lem}\label{lem:3p}
Let $f$ be a polynomial with $\deg f<p$, over a field of characteristic $p>3$,
satisfying $f(x)=-x^pf(1/x)$.
Then there is a unique $G$-invariant polynomial of formal degree $3p$ such that
$g(x)\equiv f(1-x)\pmod{x^{p+1}}$, and is given by
\[
g(x)=x^pf(x)+(1-x^p)f(1-x)+x^{2p}(1-x^p)f(1-1/x).
\]
\end{lem}

\begin{proof}
That $g$ is $G$-invariant follows by direct verification, or from a previous observation
(namely, by taking $x^pf(x)$ in place of $f$ in Equation~\eqref{eq:G-invariant}).

Our hypotheses on $f$ imply that both $0$ and $1$ are roots of $f$
(as well as $\infty$ if we assign $f$ formal degree $p$), and so
$f(x)$, $f(1-x)$ and $x^pf(1-1/x)$ are all polynomials of ordinary degree less than $p$
and without constant term.
It follows that $g(x)\equiv f(1-x)\pmod{x^{p+1}}$.

Finally, uniqueness of $f$ follows from Lemma~\ref{lem:unique}.
\end{proof}

\begin{proof}
[Proof of Equation~\eqref{eq:L^3}]
According to Lemma~\ref{lem:pounds-powers},
the polynomial
\[
g(x)=\pounds_1(x)^3/6
-(2/3)x^p(1-x^p)\pounds_3(-1)
\]
satisfies the hypotheses of Lemma~\ref{lem:3p} with $f(x)=\pounds_3(x)$.
\end{proof}

\section{Special values of $\pounds_d(x)$}\label{sec:special_values}

In this section we collect several known congruences for special values of the finite polylogarithms $\pounds_d(x)$,
and use the identities for finite polylogaritms from Section~\ref{sec:polylog}
to prove some new ones.

Let $B_n(x)$, $B_n=B_n(0)$ and $E_n$ denote the Bernoulli polynomials, and the Bernoulli and Euler numbers.
Note that $\pounds_d(1)=H_{p-1}(d)$, where $H_{k}(d)=\sum_{r=1}^{k}1/r^d$.
For any prime $p>d+2$ we have
\[
\pounds_d(1)\equiv
\begin{cases}
\displaystyle
-\frac{d(d+1)}{2(d+2)}\,p^2\,B_{p-d-2}  \pmod{p^3} &\mbox{if $d$ is odd,}\\
\displaystyle
\frac{d}{d+1}p\,B_{p-d-1}  \pmod{p^2} &\mbox{if $d$ is even.}
\end{cases}
\]
In essence these were found by Glaisher in 1900 in his articles in Quart.~J.~Math.,
but see~\cite[Theorem~5.1]{ZHS:00} for a sharper result.

To compute $\pounds_d(-1)=-H_{p-1}(d)+2^{1-d}H_{(p-1)/2}(d)$,
we combine the above congruences with the evaluation of $H_{(p-1)/2}(d)$
given in~\cite[Theorem~5.2]{ZHS:00}.
For $d=1$ we find, for any prime $p>3$, the congruence
\[
\pounds_1(-1)\equiv
-2q_p(2)+p q_p^2(2)-p^2\left(\frac{2}{3}\,q_p^3(2)+\frac{1}{4}\,B_{p-3}\right)  \pmod{p^3};
\]
for $d>1$, as soon as $p>d+1$, we obtain
\[
\pounds_d(-1)\equiv
\begin{cases}
\displaystyle
-\frac{2(1-2^{1-d})}{d}\,B_{p-d} \pmod{p} &\mbox{if $d$ is odd,}\\
\displaystyle
\frac{d(1-2^{-d})}{(d+1)}\,p\,B_{p-d-1} \pmod{p^2} &\mbox{if $d$ is even.}
\end{cases}
\]

From~\cite[Theorem 4.1]{ZHS:08b} and Equation~\eqref{C1b} we obtain congruences for $\pounds_d(2)$ and $\pounds_d(1/2)$, all valid for $p>3$:
\begin{align*}
\pounds_1(2) &\equiv -2q_p(2)-\frac{7}{12}\,p^2B_{p-3}
\pmod{p^3},\\
\pounds_2(2) &\equiv -q_p^2(2)+p\left(\frac{2}{3}\,q_p^3(2)+\frac{7}{6}\,B_{p-3}\right)
\pmod{p^2},\\
\pounds_3(2) &\equiv -\frac{1}{3}\,q_p^3(2)-\frac{7}{24}\,B_{p-3}
\pmod{p},\\
\pounds_1(1/2)
&\equiv
q_p(2)-\frac{1}{2}\,pq_p^2(2)+p^2\left(\frac{1}{3}\,q_p^3(2)-\frac{7}{48}\,B_{p-3}\right)
\pmod{p^3},\\
\pounds_2(1/2)
&\equiv
-\frac{1}{2}\,q_p^2(2)+p\left(\frac{1}{2}\,q_p^3(2)+\frac{7}{24}\,B_{p-3}\right)
\pmod{p^2},\\
\pounds_3(1/2)
&\equiv
\frac{1}{6}\,q_p^3(2)+\frac{7}{48}\,B_{p-3}
\pmod{p}.
\end{align*}
The above evaluation of $\pounds_3(2)$ appears also in~\cite{DS:06}.

Finally, according to~\cite[Corollary~2.1]{ZHS:08} combined with Fermat's little theorem,
for $d\geq 1$ and $m,r\geq 0$ we have
\[
\sum_{\substack{0<k<p\\ k\equiv r\pmod{m}}}
\!\!\!\!\!\!\frac{1}{k^d}
\equiv \frac{1}{dm^d}\left(
B_{p-d}\left(\left\{\frac{r}{m}\right\}\right)-B_{p-d}\left(\left\{\frac{r-p}{m}\right\}\right)
\right)\pmod{p},
\]
provided the prime $p$ satisfies $p>d+3$ and $p\nmid m$,
where $\{x\}=x-\lfloor x\rfloor$ is the fractional part of $x$.
The above relation can be used to compute $\pounds_d(x)$ modulo $p$ where $x$ is an $m$-th root of unity.
In particular, one finds that
\begin{align*}
\pounds_2(\pm i) &\equiv \frac{1}{16}\left(\leg{-1}{p}\pm i\right)B_{p-2}(1/4)=\frac{1}{2}\left(\leg{-1}{p}\pm i\right)E_{p-3}
\pmod{p},\\
\pounds_3(\pm i) &\equiv \frac{1}{32}\left(-1\pm\leg{-1}{p}i\right)\,B_{p-3}
\pmod{p},
\end{align*}
and
\begin{align*}
\pounds_2(\omega_6^{\pm 1}) &\equiv \frac{1}{8}\biggl(\leg{p}{3}\pm i\frac{\sqrt{3}}{3}\biggr)B_{p-2}(1/3),
&
\pounds_2(-\omega_6^{\pm 1}) &\equiv \frac{1}{12}\left(\leg{p}{3}\mp i{\sqrt{3}}\right)B_{p-2}(1/3),\\
\pounds_3(\omega_6^{\pm 1}) &\equiv \frac{1}{18}\left(1\mp i{\sqrt{3}}\leg{p}{3}\right)B_{p-3},
&
\pounds_3(-\omega_6^{\pm 1}) &\equiv \frac{2}{9}\biggl(-1\mp i\frac{\sqrt{3}}{3}\leg{p}{3}\biggr)B_{p-3},\\
\end{align*}
all four congruences being modulo $p$.

After recalling these known evaluations, we put to good use the group of transformations $G$
which we introduced in Section~\ref{sec:polylog-proofs}, Equation~\eqref{eq:G}.
Recall that its orbits on any field $F$ have all length six, with the only exceptions of
$\{1,0,\infty\}$, $\{-1,2,1/2\}$, and possibly $\{\omega_6,\omega_6^{-1}\}$
if the field contains a root $\omega_6$
of the polynomial $x^2-x+1$.
We now consider three particular orbits of lenght six, namely
\begin{gather*}
\{i,-i,1+i,1-i,(1+i)/2,(1-i)/2\},\\
\{-\omega_6,-\omega_6^{-1},1+\omega_6,1+\omega_6^{-1},(1+\omega_6)/3,(1+\omega_6^{-1})/3\},\\
\{\phi_+,\phi_-,\phi_+^2,\phi_-^2,-\phi_+,-\phi_- \},
\end{gather*}
where
$\phi_{\pm}=(1\pm\sqrt{5})/2$ are the roots of the polynomial $x^2-x-1$.
For each of these orbits the congruences given in Equations~\eqref{C1}--\eqref{C4} provide
several linear relations among the values modulo $p$ of $\pounds_2(\alpha)$
with $\alpha$ ranging over the orbit.
In the first two cases this will allow us to recover all those values
from just one which is available from the literature, and in the third case the
relations alone are sufficient to determine all those values.
At this point we need the following lemma.

\begin{lem}\label{lem:Euler}
Let $p$ be an odd prime and let $a$ be an integer not divisible by $p$. Then
\begin{equation*}
\leg{a}{p}a^{(p-1)/2}\equiv \sum_{k=0}^{n-1}\binom{1/2}{k}\bigl(p\,q_p(a)\bigr)^k \pmod{p^n}
\end{equation*}
for any positive integer $n$.
\end{lem}

\begin{proof}
\comm{Double {\em assertion}}
The assertion, whose special case $n=1$ is a familiar assertion, follows from the fact that
\[
\leg{a}{p}a^{(p-1)/2}=
\sum_{k=0}^{\infty}\binom{1/2}{k}\bigl(p\,q_p(a)\bigr)^k
\]
in the ring of $p$-adic integers $\Z_p$.
The latter is true because both sides are square roots of the integer $a^{p-1}=1+p\,q_p(a)$ in $\Z_p$,
and both are congruent to $1$ modulo $p$.
\end{proof}

\begin{thm}\label{thm:i}
For any prime $p>3$ we have
\begin{align*}
\pounds_2(1{\pm}i)&\equiv
-\frac{q_p^2(2)}{8}\,\left(1\pm i\leg{-1}{p}\right)+\frac{1}{2}\leg{-1}{p}E_{p-3} \pmod{p},\\
\pounds_2\bigl((1{\pm}i)/2\bigr)&\equiv
-\frac{q_p^2(2)}{8}\,+\frac{1}{4}\left(\leg{-1}{p}\pm i\right)E_{p-3} \pmod{p}.
\end{align*}
\end{thm}

\begin{proof}
We first compute $\pounds_2(1{\pm}i)=a\pm ib$,
from which the remaining values can be obtained by means of the inversion relation~\eqref{C1}.
According to~\cite[Theorem 3.2]{ZHS:08b} we have
\begin{align*}
\mbox{Re}\bigl(\pounds_1(i)\bigr)&=\sum_{k=1}^{\lfloor p/4\rfloor}\frac{1}{4k}-\sum_{k=1}^{\lfloor p/4\rfloor}\frac{1}{4k-2}
=\frac{1}{2}\sum_{k=1}^{\lfloor p/4\rfloor}\frac{1}{k}-\frac{1}{2}\sum_{k=1}^{\lfloor p/2\rfloor}\frac{1}{k}\\
&\equiv -\frac{1}{2}\,q_p(2)+\frac{1}{4}\,p\,q_p^2(2)-\frac{1}{2}\,p\leg{-1}{p}\,E_{p-3} \pmod{p^2}.
\end{align*}
Because
\[
(1\pm i)^n=(-1)^{(n^2-1)/8} 2^{(n-1)/2}\left(1\pm(-1)^{(n-1)/2}i\right)
\]
for $n$ odd, Lemma~\ref{lem:Euler} implies
\begin{align*}
\mbox{Re}\bigl(Q_p(1-i)\bigr)
&=
\frac{\mbox{Re}\bigl((1-i)^p\bigr)-1}{p}
=
\frac{\leg{2}{p}2^{(p-1)/2}-1}{p}
\\&
\equiv
\frac{1}{2}\,q_p(2)-\frac{1}{8}\,p\,q^2_p(2)\pmod{p^2}.
\end{align*}
Using Equation~\eqref{C2} we find
\[
\mbox{Re}\bigl(Q_{p}(1-i)\bigr)\equiv -\mbox{Re}\bigl(\pounds_1(i)\bigr)-p\,a\pmod{p^2},
\]
which allows us to determine $a$.
Finally, Equation~\eqref{C3} implies
\[
\frac{1}{2}\leg{-1}{p}E_{p-3}
\equiv \mbox{Re}\bigl(\pounds_2(i)\bigr)
\equiv \mbox{Re}\bigl(a-ib+i^p(a+ib)\bigr)\equiv a-\leg{-1}{p}b\pmod{p},
\]
which yields $b$.
\end{proof}

\begin{thm}\label{thm:omega}
For any prime $p>3$ we have
\begin{align*}
\pounds_2(1+\omega_6^{\pm 1})&\equiv
-\frac{q_p^2(3)}{16}\,\left(3\pm i\sqrt{3}\leg{p}{3}\right)+\frac{1}{36}\left(3\leg{p}{3}\mp i\sqrt{3}\right)B_{p-2}(1/3) \pmod{p},\\
\pounds_2\bigl((1+\omega_6^{\pm 1})/3\bigr)&\equiv
-\frac{q_p^2(3)}{8}+\frac{1}{36}\left(\leg{p}{3}\pm i\sqrt{3}\right)B_{p-2}(1/3) \pmod{p}.
\end{align*}
\end{thm}

\begin{proof}
We compute $\pounds_2(1+\omega_6^{\pm 1})=a\pm ib$,
and the other congruence will follow from the inversion relation~\eqref{C1}.
From~\cite[Theorem 3.9]{ZHS:08b} we have
\begin{align*}
\mbox{Re}\bigl(\pounds_1(-\omega_6)\bigr)
&=\frac{3}{2}\sum_{k=1}^{\lfloor p/3\rfloor}\frac{1}{3k}-\frac{1}{2}\sum_{k=1}^{p-1}\frac{1}{k}\\
&\equiv -\frac{3}{4}\,q_p(3)+\frac{3}{8}\,p\,q_p^2(3)-\frac{1}{12}\,p\leg{p}{3}\,B_{p-2}(1/3) \pmod{p^2}.
\end{align*}
Because
\[
\omega_6^{\pm n}=\frac{(-1)^{n-1}}{2} \left(1\pm i\sqrt{3}\leg{n}{3}\right)
\]
if $n$ is not a multiple of $3$,
Lemma~\ref{lem:Euler} implies
\begin{align*}
\mbox{Re}\bigl(Q_p(1+\omega_6)\bigr)
&=
\frac{\mbox{Re}\bigl((\sqrt{3}\,i\,\omega_6^{-1})^p\bigr)-\mbox{Re}(\omega_6^p)-1}{p}
=\frac{3}{2}\,\frac{\leg{3}{p}3^{(p-1)/2}-1}{p}
\\&\equiv
\frac{3}{4}\,q_p(3)-\frac{3}{16}\,p\,q^2_p(3)\pmod{p^2}.
\end{align*}
Using Equation~\eqref{C2} we find
\[
\mbox{Re}\bigl(Q_{p}(1+\omega_6)\bigr)\equiv -\mbox{Re}\bigl(\pounds_1(-\omega_6)\bigr)-p\,a\pmod{p^2},
\]
from which we can determine $a$.
Finally, Equation~\eqref{C3} implies
\[
\frac{1}{12}\leg{p}{3}B_{p-2}(1/3)
\equiv \mbox{Re}\bigl(\pounds_2(-\omega_6)\bigr)
\equiv \mbox{Re}\bigl(a+ib-\omega_6^p(a-ib)\bigr)
\equiv \frac{1}{2}\left(a-\sqrt{3}\leg{p}{3}\,b\right)\pmod{p},
\]
which yields $b$.
\end{proof}

\begin{thm}\label{thm:phi}
For any prime $p>5$ we have
\begin{align*}
\pounds_2(\phi_{\pm})&\equiv
\mp \frac{\sqrt{5}}{10}\leg{p}{5}q_L^2\pmod{p},\\
\pounds_2(\phi^2_{\pm})&\equiv -\frac{1}{2}\biggl(1\pm \frac{\sqrt{5}}{5}\leg{p}{5}\biggr)q_L^2 \pmod{p},\\
\pounds_2(-\phi_{\pm})&\equiv -\frac{1}{4}\biggl(1\pm \frac{\sqrt{5}}{5}\leg{p}{5}\biggr)q_L^2 \pmod{p},
\end{align*}
where $q_L=Q_p(\phi_{\pm})=(L_p-1)/p$ is the {\em Lucas quotient}.
Moreover, we have
\[
\pounds_3(\phi_{\pm}^2)\equiv -\frac{2}{15}\left(1\pm\sqrt{5}\leg{p}{5}\right)\left(\frac{1}{2}q_L^3 +B_{p-3}\right)
\pmod{p}.
\]
\end{thm}

\begin{proof}
The distribution relation~\eqref{C6} with $m=2$ and $d=2$ yields
\[
\pounds_2(\phi_+^2)\equiv 2\phi_+^{2p}\pounds_2(\phi_+)+2\phi_+^{p}\pounds_2(-\phi_+)\pmod{p}.
\]
Equation~\eqref{C3} and the inversion relation~\eqref{C1} yield
\[
\pounds_2(\phi_+^2) -\pounds_2(-\phi_+)\equiv \phi_+^{2p}\pounds_2(-\phi_-)\equiv \phi_+^{p}\pounds_2(\phi_+) \pmod{p}.
\]
Equation~\eqref{C4} and the inversion relation~\eqref{C1} yield
\[
\frac{1}{2}q_L^2 +\phi_+^p\pounds_2(\phi_+)\equiv -\phi_-^p\pounds_2(\phi_-)\equiv -\phi_-^{2p}\pounds_2(-\phi_+)\pmod{p}.
\]
Solving the linear system for $\pounds_2(\phi_+)$, $\pounds_2(\phi^2_+)$ and $\pounds_2(-\phi_+)$  given by the above three congruences, and using
\[
2\phi_{\pm}^p\equiv \left(1\pm\sqrt{5}\leg{p}{5}\right)\quad\mbox{and}\quad
2\phi_{\pm}^{2p}\equiv \left(3\pm\sqrt{5}\leg{p}{5}\right) \pmod{p},
\]
one obtains the three stated congruences involving $\pounds_2$.

In a similar way one evaluates $\pounds_2$ and $\pounds_3$ at $\phi_{\pm}^2$.
The distribution relation~\eqref{C6}, with $m=2$ and $d=3$, combined with the inversion relation~\eqref{C1}, yields
\[
\pounds_3(\phi_+^2)-4\phi_+^{2p}\pounds_3(\phi_+)\equiv -4\pounds_3(\phi_-)\pmod{p}.
\]
Also, congruence~\eqref{C5} yields
\[
\frac{1}{6}q_L^3 +\frac{1}{3}B_{p-3}\equiv -\phi_-^p\pounds_3(\phi_-)-\phi_+^p\pounds_3(\phi_+)
+\phi_-^{p}\pounds_3(\phi_+^2) \pmod{p}.
\]
Solving for $\pounds_3(\phi_{+}^2)$ we find
\[
\pounds_3(\phi_{+}^2)\equiv-\frac{4\phi_{+}^p}{15}\left(\frac{1}{2}q_L^3 +B_{p-3}\right)
\equiv -\frac{2}{15}\left(1+\sqrt{5}\leg{p}{5}\right)\left(\frac{1}{2}q_L^3 +B_{p-3}\right)\pmod{p}.
\]
The analogous congruence for $\pounds_3(\phi_{-}^2)$ is obtained by interchanging the subscripts $+$ and $-$ throughout the proof.
\end{proof}

\section{Polynomial identities}\label{sec:identities}

The main goal of this section,
which we achieve in Theorem~\ref{thm:recurrences},
is to obtain identities which allow one to replace the two general partial sums
$\sum_{k=1}^n k^{-s}\binom{2k}{k}^{-1}t^k$, with $s=1,2,3$ (and also higher, in principle),
with more manageable sums.
Those involve the familiar Lucas sequences $\{u_n(x)\}_{n\geq 0}$ and $\{v_n(x)\}_{n\geq 0}$
defined by the recurrence relations
\begin{align*}
u_0(x)&=0, &u_1(x)&=1, &\mbox{and}\quad u_n(x)&=x\,u_{n-1}(x)-u_ {n-2}(x)
\quad\mbox{for $n>1$,}\\
v_0(x)&=2, &v_1(x)&=x, &\mbox{and}\quad v_n(x)&=x\,v_{n-1}(x)-v_ {n-2}(x)
\quad\mbox{for $n>1$.}
\end{align*}
They have generating functions
\[
U(z)=\sum_{n\ge 0}u_n(x)z^n=\frac{z}{1-xz+z^2},
\quad\text{and}\quad
V(z)=\sum_{n\ge 0}v_n(x)z^n=\frac{2-xz}{1-xz+z^2},
\]
where we have omitted the dependence of $U(z)$ and $V(z)$ on $x$ in favour of a lighter notation.
It is convenient to view $x$ as an indeterminate rather than a specific number.
Thus, letting $\alpha$ be an element of a quadratic field extension
of the field $\Q(x)$ of rational functions with $\alpha^2-x\alpha+1=0$,
we have
$u_n(x)=(\alpha^n-\alpha^{-n})/(\alpha-\alpha^{-1})$
and
$v_n(x)=\alpha^n+\alpha^{-n}$.
We anticipate that in Section~\ref{sec:congruences} we will obtain polynomial congruences, relative to a prime $p$, which involve sums of the form
$
\sum_{k=1}^{p-1}u_k(x)/k^d
=
\bigl(\pounds_d(\alpha)-\pounds_d(\alpha^{-1})\bigr)/(\alpha-\alpha^{-1})
$
and
$
\sum_{k=1}^{p-1}v_k(x)/k^d
=
\pounds_d(\alpha)+\pounds_d(\alpha^{-1})
$.
Thus, specializations of those polynomial congruences to numerical congruences will follow from knowledge
of special values of finite polylogarithms which we have obtained in Section~\ref{sec:special_values},
as we will illustrate in our final Section~\ref{sec:ZWS}.

Note that $u_{n+1}(x)$ and $v_n(x)$ are even polynomials if $n$ is even, and odd polynomials otherwise.
Some readers of different backgrounds may recognize them as related to the classical Chebyshev
polynomials of the first and second kind $T_n(x)$ and $U_n(x)$, or to their (renormalized) generalizations known as Dickson polynomials
$D_n(x,\alpha)$ and $E_n(x,\alpha)$.
In fact,
\[
u_{n+1}(x)=E_n(x,1)=U_n(x/2),
\quad\text{and}\quad
v_n(x)=D_n(x,1)=2\,T_n(x/2).
\]
The same readers may be aware that
$(d/dx)T_n(x)=n\,U_{n-1}(x)$ for $n>0$,
which becomes $(d/dx)v_n(x)=n\,u_n(x)$ here.
Besides recalling this fact in an integral formulation which is more suitable for us, the following preliminary result provides us with an expression for
a primitive of the polynomial $\bigl(v_n(x)-v_n(-2)\bigl)/(x+2)$.

\begin{lem}\label{lem:integrate}
For any $n>0$ we have
\begin{align}
\label{S1}
\int_0^t u_n(\tau-2)\,d\tau
&=\frac{v_n(t-2)-2(-1)^n}{n},\\
\label{S2}
\int_0^t \frac{v_n(\tau-2)-2(-1)^n}{\tau}\,d\tau
&=\frac{v_n(t-2)-2(-1)^n}{n}
+2\sum_{k=1}^{n-1}(-1)^{n-k}\,\frac{v_k(t-2)-2(-1)^k}{k}.
\end{align}
\end{lem}

\begin{proof}
Temporarily viewing $x$ as a complex constant and working in the formal power series ring $\C[[z]]$, write $z^2-xz+1=(1-\alpha z)(1-\beta z)$.
Then $V(z)=(1-\alpha z)^{-1}+(1-\beta z)^{-1}$, and so for $n>0$ we have
\[
[z^n]\log(z^2-xz+1)
=[z^n]\bigl(\log(1-\alpha z)+\log(1-\beta z)\bigr)
=-(\alpha^n+\beta^n)/n
=-v_n(x)/n.
\]
Using this and setting $x=\tau-2$ we obtain
\begin{align*}
\int_0^t u_n(\tau-2)\,d\tau
&=[z^n]\int_0^t U(z)\,dt\\
&=[z^n]\bigl(-\log(z^2-(t-2)z+1)+2\log(1+z)\bigr)=\frac{v_n(t-2)-2(-1)^n}{n}.
\end{align*}
Similarly, but with a slightly more complicated integrand, we obtain
\begin{align*}
\int_0^t \frac{v_n(\tau-2)-2(-1)^n}{\tau}\,d\tau
&=[z^n]\int_0^t \left(V(z)-\frac{2}{1+z}\right)\,\frac{d\tau}{\tau}\\
&=[z^n]\left(\bigl(-\log(z^2-(t-2)z+1)+2\log(1+z)\bigr)\cdot\frac{1-z}{1+z}\right)\\
&=\frac{v_n(t-2)-2(-1)^n}{n}
+2\sum_{k=1}^{n-1}(-1)^{n-k}\,\frac{v_k(t-2)-2(-1)^k}{k},
\end{align*}
where we have expanded
$(1-z)/(1+z)=1-2z/(1+z)
=1+2\sum_{k>0}(-z)^k$
for the last passage.
\end{proof}

We are now ready to state the main result of this section, which expresses certain sums of the form
$\sum_{k=1}^n k^{-s}\binom{2k}{k}^{-1}t^k$
in terms of other sums involving our Lucas sequences.
The crucial case is Equation~\eqref{I3}, where $s=1$, from which the other equations will follow by integration using Lemma~\ref{lem:integrate}.
The case $s=0$ excluded here may be obtained from Equation~\eqref{I3} by differentiation,
but we prefer to deal with it differently in Theorem~\ref{thm:identities}.

\begin{thm}\label{thm:recurrences}
For $n\geq 1$ we have the polynomial identities
\begin{align}
\label{I3}
\binom{2n}{n}\sum_{k=1}^n \frac{t^{k-1}}{k\binom{2k}{k}}
&=
\sum_{k=1}^{n} \binom{2n}{n-k}\frac{u_{k}(t-2)}{k},\\
\label{I4}
\binom{2n}{n}\sum_{k=1}^n \frac{t^k}{k^2\binom{2k}{k}}
&=
\sum_{k=1}^{n} \binom{2n}{n-k}\frac{v_{k}(t-2)}{k^2}+\binom{2n}{n}\sum_{k=1}^n \frac{1}{k^2},\\
\label{I4b}
\binom{2n}{n}\sum_{k=1}^n \frac{t^k}{k^3\binom{2k}{k}}
&=
\sum_{k=1}^{n} \binom{2n}{n-k}\frac{v_{k}(t-2)}{k^3}\\ \nonumber
&\;\;\; +2\sum_{1\leq j<k\leq n}\binom{2n}{n-k}\frac{(-1)^{k-j}\,v_{j}(t-2)}{jk^2}
+\binom{2n}{n}\sum_{k=1}^n \frac{1}{k^3}.
\end{align}
\end{thm}

Our proof of Equation~\eqref{I3} involves a transformation of sequences given by
\[
\{c(n)\}_{n\geq 1}\to \{s(n)\}_{n\geq 0},
\quad\mbox{where}\quad
s(n)=\binom{2n}{n}\sum_{k=1}^n \frac{c(k)}{\binom{2k}{k}},
\]
which we read as $s(0)=0$ for $n=0$.
More generally, in the sequel we interpret a sum to vanish when the upper summation limit is one less than the lower summation limit.
The resulting sequence $s(n)$ is related to the original sequence $c(n)$ by the recurrence
\begin{equation}\label{RE}
s(0)=0,\quad \Delta_n(s(n)):=(n+1)\,s(n+1)-2(2n+1)\,s(n)=(n+1)\,c(n+1).
\end{equation}

\begin{proof}[Proof of Theorem~\ref{thm:recurrences}]
We start with proving Equation~\eqref{I3}.
Consider the sequence
\[
a_d(n)=\binom{2n}{n}\sum_{k=1}^n \frac{t^{k-1}}{k^d \binom{2k}{k}},
\]
and the corresponding generating function $A_d(z)=\sum_{n\geq 0}a_d(n)z^n$.
When $d=1$ we have
\begin{align*}
\bigl(A_1(z)\sqrt{1-4z}\bigr)'
&=\sum_{n=0}^{\infty}\left(na_1(n)z^{n-1}\sqrt{1-4z}-{2a_1(n)z^n\over \sqrt{1-4z}}\right)\\
&={1\over \sqrt{1-4z}}\sum_{n=0}^{\infty}\left(na_1(n)z^{n-1}-4na_1(n)z^{n}-2a_1(n)z^n\right)\\
&={1\over \sqrt{1-4z}}\biggl(\sum_{n=1}^{\infty}na_1(n)z^{n-1}-2\sum_{n=0}^{\infty}\left(2n+1\right)a_1(n)z^n\biggr)\\
&={1\over \sqrt{1-4z}}\sum_{n=0}^{\infty}\bigl((n+1)a_1(n+1)-2(2n+1)a_1(n)\bigr)z^n.
\end{align*}
According to Equation~\eqref{RE} we have
\[
\Delta_n\bigl(a_1(n)\bigr)=(n+1)a_1(n+1)-2(2n+1)a_1(n)=
t^n\quad\mbox{for $n\geq 0$},
\]
and so
\[
\bigl(A_1(z)\sqrt{1-4z}\bigr)'={1\over \sqrt{1-4z}}\left(\sum_{n=0}^{\infty}(tz)^n\right)=\frac{1}{(1-tz)\sqrt{1-4z}}.
\]
Now consider the sequence
\[
b_1(n)=\sum_{k=1}^{\infty} \binom{2n}{n+k}\frac{u_{k}(t-2)}{k}.
\]
Its generating function $B_1(z)=\sum_{n\geq 0}b_1(n)z^n$ is
\begin{align*}
B_1(z)
&=\sum_{k=1}^{\infty} \frac{u_{k}(t-2)}{k}\sum_{n\geq 1}\binom{2n}{n+k}z^n\\
&=\sum_{k=1}^{\infty} \frac{u_{k}(t-2)}{k}\left(\frac{4z}{(1+\sqrt{1-4z})^2}\right)^k\frac{1}{\sqrt{1-4z}}
=\frac{1}{\sqrt{1-4z}}\, U_1(h(z)),
\end{align*}
where
\[
h(z)=\frac{4z}{(1+\sqrt{1-4z})^2},
\quad\mbox{and}\quad
U_d(z)=\sum_{k=1}^{\infty} \frac{u_{k}(t-2)z^k}{k^d}.
\]
Because $z(d/dz)U_1=U$, we deduce that
\[
\bigl(B_1(z)\sqrt{1-4z}\bigr)'
=\frac{d}{dz}U_1\bigl(h(z)\bigr)
=\frac{h'(z)}{h(z)}\;U\bigl(h(z)\bigr)
=\frac{1}{(1-tz)\sqrt{1-4z}}.
\]
Finally, $A_1(0)=B_1(0)$ and $\bigl(A_1(z)\sqrt{1-4z}\bigr)'=\bigl(B_1(z)\sqrt{1-4z}\bigr)'$
imply that $A_1(z)=B_1(z)$, and we conclude that Equation~\eqref{I3} holds.

To prove Equation~\eqref{I4}, integrate Equation~\eqref{I3} with respect to $t$ and then use Equation~\eqref{S1}, to obtain
\begin{align*}
\binom{2n}{n}\sum_{k=1}^n \frac{t^k}{k^2\binom{2k}{k}}
&=\sum_{k=1}^{n} \binom{2n}{n-k}\frac{v_{k}(t-2)-2(-1)^k}{k^2}\\
&=\sum_{k=1}^{n} \binom{2n}{n-k}\frac{v_{k}(t-2)}{k^2}+\binom{2n}{n}\sum_{k=1}^n \frac{1}{k^2}.
\end{align*}

One can prove Equation~\eqref{I4b} in a similar way,
by integrating Equation~\eqref{I4} divided by $t$ and then using Equation~\eqref{S2}.
\end{proof}

Equation~\eqref{I5} in the following result shows how the study of
$\sum_{k=1}^n \binom{2k}{k}^{-1}t^k$
can be reduced to the sums considered in Theorem~\ref{thm:recurrences}.
Equation~\eqref{I6} gives a similar formula for
$\sum_{k=1}^n H_{k-1}(s)\binom{2k}{k}^{-1}t^k$ with $s>0$.
Note that Equation~\eqref{I6} would not specialize correctly to the case $s=0$, where $H_{k-1}(0)=k-1$,
which instead may be obtained from Equation~\eqref{I4} by differentiation if one wishes.

\begin{thm}\label{thm:identities}
For any $n,s\geq 1$ we have the polynomial identities
\begin{align}
\label{I5}
(t-4)\sum_{k=1}^n \frac{t^{k-1}}{\binom{2k}{k}}+
2\sum_{k=1}^n \frac{t^{k-1}}{k\binom{2k}{k}}
&=
\frac{t^n}{\binom{2n}{n}}-1,\\
\label{I6}
(t-4)\sum_{k=1}^n \frac{t^{k-1}H_{k-1}(s)}{\binom{2k}{k}}+
2\sum_{k=1}^n \frac{t^{k-1}H_{k-1}(s)}{k\binom{2k}{k}}
&=
\frac{t^nH_n(s)}{\binom{2n}{n}}-\sum_{k=1}^n \frac{t^{k}}{k^s\binom{2k}{k}}.
\end{align}
\end{thm}
\begin{proof}
With the same notation as in the proof of Theorem~\ref{thm:recurrences}, Equation~\eqref{RE} implies
\[
\Delta_n\bigl((t-4)a_0(n)+2a_1(n)\bigr)=(t-4)(n+1)t^n+2t^n=
\Delta_n(t^n)=\Delta_n\left(t^n-\binom{2n}{n}\right).
\]
Because the two sequences agree on $n=0$, Equation~\eqref{I5} follows.

To prove Equation~\eqref{I6}, consider
\[
a^{(s)}_d(n)=\binom{2n}{n}\sum_{k=1}^n \frac{t^{k-1}H_{k-1}(s)}{k^d \binom{2k}{k}}.
\]
Equation~\eqref{RE} yields, for $n\geq 0$,
\[
\Delta_n(a^{(s)}_d(n))=\frac{t^nH_n(s)}{(n+1)^{d-1}}.
\]
This implies
\begin{align*}
\Delta_n\left((t-4)a^{(s)}_0(n)+2a^{(s)}_1(n)\right)
&=(t-4)(n+1)t^nH_n(s)+2t^nH_n(s)\\
&=\Delta_n \bigl(t^nH_n(s)\bigr)-\frac{t^{n+1}}{(n+1)^{s-1}} \\
&=\Delta_n \bigl(t^nH_n(s)-t\,a_s(n)\bigr).
\end{align*}
Because the two sequences agree on $n=0$, Equation~\eqref{I6} follows.
\end{proof}

We point out that trigonometric versions of our Equations~\eqref{I3}, \eqref{I4} and~\eqref{I5}, with $4\cos^2\varphi$ in place of $t$,
have recently appeared in~\cite[Equations~(1.1), (5.1) and (1.3)]{WS:08}.
The proofs given there are essentially different from ours.

\section{Polynomial congruences}\label{sec:congruences}

In this section we specialize the two partial sums
$\sum_{k=1}^n k^{-s}\binom{2k}{k}^{-1}t^k$, with $s=1,2$,
considered in Theorem~\ref{thm:recurrences}, by setting $n=p-1$, and study their values modulo $p^2$.
(Note that the values of those sums become $p$-integral only upon multiplication by $p$.)
Theorem~\ref{thm:congruences} also contains similar but less sharp evaluations for the corresponding sums
$\sum_{k=1}^n k^{-s}H_{k-1}(2)\binom{2k}{k}^{-1}t^k$.
As we anticipated in the first paragraph of Section~\ref{sec:identities},
the possibility of specializing these polynomial congruences to numerical congruences,
exemplified in Section~\ref{sec:ZWS},
depends on our knowledge of special values of finite polylogarithms which we have developed in Section~\ref{sec:special_values}.

\begin{thm}\label{thm:congruences}
For any prime $p>3$ we have the polynomial congruences
\begin{align}
\label{CC1}
p\sum_{k=1}^{p-1} \frac{t^{k}}{k\binom{2k}{k}}
&\equiv
\frac{tu_p(2-t)-t^p}{2}+p^2\, t\sum_{k=1}^{p-1} \frac{u_{k}(2-t)}{k^2} \pmod{p^3},\\
\label{CC2}
p\sum_{k=1}^{p-1} \frac{t^{k}}{k^2\binom{2k}{k}}
&\equiv
\frac{2-v_p(2-t)-t^p}{2p}-p^2\, \sum_{k=1}^{p-1} \frac{v_{k}(2-t)}{k^3} \pmod{p^3},
\end{align}
and also
\begin{align}
\label{CC3}
p\sum_{k=1}^{p-1} \frac{t^{k}H_{k-1}(2)}{k\binom{2k}{k}}
&\equiv
t\sum_{k=1}^{p-1} \frac{u_{k}(2-t)}{k^2} \pmod{p},\\
\label{CC4}
p\sum_{k=1}^{p-1} \frac{t^{k}H_{k-1}(2)}{k^2\binom{2k}{k}}
&\equiv
-\sum_{k=1}^{p-1} \frac{v_{k}(2-t)}{k^3} \pmod{p}.
\end{align}
\end{thm}

\begin{proof}
Setting $n=p$ in Equation~\eqref{I3} and multiplying by $pt$ we obtain
\[
p\binom{2p}{p}\sum_{k=1}^{p-1} \frac{t^{k}}{k\binom{2k}{k}}+{t^p}
=tu_{p}(t-2)+p\,t\sum_{k=1}^{p-1} \binom{2p}{k}\frac{u_{p-k}(t-2)}{p-k}.
\]
Now we use the standard congruences
$\binom{2p}{k}\equiv 2(-1)^{k-1}p/k\pmod{p^2}$,
for $k=1,\dots,p-1$, and
$\binom{2p}{p}\equiv 2-\frac{4}{3}\,p^3B_{p-3}\pmod{p^4}$.
Because $u_k(-x)=(-1)^{k-1}u_k(x)$, we deduce
\begin{align*}
2p\sum_{k=1}^{p-1} \frac{t^{k}}{k\binom{2k}{k}}+{t^p}
&\equiv tu_{p}(t-2)+2p^2t\sum_{k=1}^{p-1} \frac{(-1)^{k-1}u_{p-k}(t-2)}{k(p-k)}\nonumber&\\
&\equiv tu_{p}(2-t)+2p^2t\sum_{k=1}^{p-1}\frac{u_{k}(2-t)}{k^2}\pmod{p^3},
\end{align*}
which is equivalent to the desired Equation~\eqref{CC1}.

To pass from this to Equation~\eqref{CC3}
we need to relate the sums
$\sum_{k=1}^{p-1}k^{-1}\binom{2k}{k}^{-1}t^k$
and
$\sum_{k=1}^{p-1}k^{-1}H_{k-1}(2)\binom{2k}{k}^{-1}t^k$
via an appropriate congruence.
To this purpose
we need the former of the following identities, valid for $n\geq 1$,
which were obtained by the second author in the course of the proof of~\cite[Theorem 3.1]{Ta:10}:
\begin{align}
\label{I1}
\sum_{k=1}^n \binom{n}{k}\binom{n+k-1}{k-1}\frac{(-t)^{k-1}}{\binom{2k}{k}}
&=
\frac{(-1)^{n-1}u_n(t-2)}{2},\\
\label{I2}
\sum_{k=0}^n \binom{n}{k}\binom{n+k-1}{k}\frac{(-t)^k}{\binom{2k}{k}}
&=
\frac{(-1)^n v_n(t-2)}{2}.
\end{align}
The latter identity will be needed later to pass from Equation~\eqref{CC2} to Equation~\eqref{CC4}.
Note that the coefficient of $(-t)^{k-1}$
in the former formula, for example, may be more simply written as
$\frac{1}{2}\binom{n+k-1}{2k-1}$,
but here we need the longer form, with the factor $\binom{2k}{k}^{-1}$ in evidence.

Thus, setting $n=p$ in Equation~\eqref{I1} and separating the last summand we obtain
\[
\sum_{k=1}^{p-1}\binom{p}{k}\binom{p-1+k}{k}\frac{(-t)^k}{\binom{2k}{k}}
=-\frac{tu_p(2-t)-t^p}{2}.
\]
One easily checks that for $k=1,\dots,p-1$ we have
\begin{align*}
\frac{k}{p}\binom{p}{k}
=\binom{p-1}{k-1}
&\equiv
(-1)^{k-1}\left(1-pH_{k-1}(1)+p^2H_{k-1}(1,1)\right)
\pmod{p^3},\\
\binom{p-1+k}{k-1}
&\equiv
1+pH_{k-1}(1)+p^2H_{k-1}(1,1)
\pmod{p^3},
\end{align*}
whence
\begin{align*}
\binom{p}{k}\binom{p-1+k}{k-1}
&\equiv (-1)^{k-1}
\frac{p}{k}\left(1-p^2\left(H_{k-1}(1)^2-2H_{k-1}(1,1)\right)\right)&\\
&\equiv (-1)^{k-1}\frac{p}{k}\left(1-p^2H_{k-1}(2)\right)
\pmod{p^4}.
\end{align*}
Noting that $\binom{2k}{k}$ can be a multiple of $p$ but not of $p^2$ in the range considered, we obtain
\[
p\sum_{k=1}^{p-1}\frac{t^k}{k\binom{2k}{k}}
-p^3\sum_{k=1}^{p-1}\frac{t^k H_{k-1}(2)}{k\binom{2k}{k}}\equiv\frac{tu_p(2-t)-t^p}{2} \pmod{p^3}.
\]
Together with Equation~\eqref{CC1} this implies Equation~\eqref{CC3}.

The proofs of Equations~\eqref{CC2} and~\eqref{CC4} are similar.
Setting $n=p$ in Equation~\eqref{I4} and multiplying by $p$ we obtain
\[
p\binom{2p}{p}\sum_{k=1}^{p-1} \frac{t^{k}}{k^2\binom{2k}{k}}+\frac{t^p}{p}
=\frac{v_{p}(t-2)}{p}+p\sum_{k=1}^{p-1} \binom{2p}{k}\frac{v_{p-k}(t-2)}{(p-k)^2}+p\binom{2p}{p}H_{p-1}(2)
+\frac{1}{p}\binom{2p}{p}.
\]
Because $v_k(-x)=(-1)^{k}v_k(x)$ and
$H_{p-1}(2)\equiv 0\pmod{p}$,
we have
\begin{align*}
2p\sum_{k=1}^{p-1} \frac{t^{k}}{k^2\binom{2k}{k}}
&\equiv \frac{2-v_{p}(t-2)-t^p}{p}+2p^2\sum_{k=1}^{p-1} \frac{(-1)^{k-1}v_{p-k}(t-2)}{k(p-k)^2}\\
&\equiv \frac{2-v_p(2-t)-t^p}{p}-2p^2\, \sum_{k=1}^{p-1} \frac{v_{k}(2-t)}{k^3}
\pmod{p^3}.
\end{align*}
and hence Equation~\eqref{CC2} holds.

Setting $n=p$ in Equation~\eqref{I2}, dividing by $p$ and separating the last summand we find
\[
\frac{1}{p}\sum_{k=1}^{p-1} \binom{p}{k}\binom{p-1+k}{k}\binom{2k}{k}^{-1}(-t)^k=\frac{2-v_p(t-2)-t^p}{2p}.
\]
In the range considered for $k$ we have
\[
\frac{1}{p}\binom{p}{k}\binom{p-1+k}{k}=\frac{1}{k}\binom{p}{k}\binom{p-1+k}{k-1}
\equiv (-1)^{k-1}\frac{p}{k^2}\left(1-p^2H_{k-1}(2)\right) \pmod{p^4},
\]
and $\binom{2k}{k}$ is not a multiple of $p^2$, and hence
\[
p\sum_{k=1}^{p-1}\frac{t^k}{k^2\binom{2k}{k}}
-p^3\sum_{k=1}^{p-1}\frac{t^k H_{k-1}(2)}{k^2\binom{2k}{k}}\equiv\frac{2-v_p(t-2)-t^p}{2p} \pmod{p^3}.
\]
Together with Equation~\eqref{CC2} this implies Equation~\eqref{CC4}.
\end{proof}

Theorem~\ref{thm:congruences} has exploited only the first two of the three polynomial congruences produced in Theorem~\ref{thm:recurrences}.
The third congruence we can only use in a weakened form, obtaining the following result.

\begin{thm}\label{thm:congruences3}
For any prime $p>3$ we have the polynomial congruence
\begin{equation*}
p\sum_{k=1}^{p-1} \frac{t^k}{k^3\binom{2k}{k}}
\equiv
\frac{1-\bigl(v_p(2-t)+t^p\bigr)\binom{2p}{p}^{-1}}{p^2}-\frac{1}{p}\sum_{k=1}^{p-1} \frac{v_{k}(2-t)}{k} \pmod{p^2}.
\end{equation*}
\end{thm}

\begin{proof}
The proof runs along similar lines as that of Theorem~\ref{thm:congruences},
but starting from Equation~\eqref{I4b}.
\end{proof}

\section{Congruences with $\binom{2k}{k}$ in the numerators}\label{sec:numerators}

In this section we prove polynomial identities and congruences for sums
similar to those considered in the previous sections,
but involving the central binomial coefficients $\binom{2k}{k}$ in the numerators rather than the denominators.

One can obtain polynomials identities analogous to those of Section~\ref{sec:identities}
starting from the identity
\begin{equation}\label{I7}
\sum_{k=0}^{n-1} \binom{2k}{k} t^{n-1-k}=
\sum_{k=1}^{n} \binom{2n}{n-k}u_{k}(t-2),
\end{equation}
which was proved in~\cite{ST:10}.
In fact, successive integration according to Lemma~\ref{lem:integrate}
produces the polynomial identities
\begin{align}
\label{eq:S1}
\sum_{k=0}^{n-1} \frac{\binom{2k}{k}}{n-k}\, t^{n-k}&=
\sum_{k=1}^{n} \binom{2n}{n-k}\frac{v_{k}(t-2)-2(-1)^k}{k},
\\
\label{eq:S2}
\sum_{k=0}^{n-1} \frac{\binom{2k}{k}}{(n-k)^2}\, t^{n-k}&=
\sum_{k=1}^{n} \binom{2n}{n-k}\frac{v_{k}(t-2)-2(-1)^k}{k^2}\\\nonumber
&\quad
+2\sum_{1\leq j<k\leq n} \binom{2n}{n-k}\frac{(-1)^{k-j}\bigl(v_{j}(t-2)-2(-1)^j\bigr)}{jk},
\end{align}
which are somehow analogous to the first two identities in Theorem~\ref{thm:recurrences}.
Equation~\eqref{eq:S1} will play a role in deducing Equation~\eqref{CC9} from Equation~\eqref{CC8} in our proof of Theorem~\ref{thm:last} below.

Passing now to polynomial congruences, a simple way of switching central binomial coefficients
from denominators to numerators of our sums is based on the congruence
\[
\frac{2p}{k\binom{2k}{k}}\equiv\binom{2(p-k)}{p-k}\pmod{p},
\qquad\text{for $k=1,\dots, p-1$.}
\]
Accordingly, Equations~\eqref{CC3} and~\eqref{CC4} of Theorem~\ref{sec:congruences} have equivalent formulations
\begin{align}
\label{CC5}
\sum_{k=1}^{p-1} {t^{p-k}H_{k}(2)}\binom{2k}{k}
&\equiv
-2t\sum_{k=1}^{p-1} \frac{u_{k}(2-t)}{k^2} \pmod{p},\\
\label{CC6}
\sum_{k=1}^{p-1} \frac{t^{p-k}H_{k}(2)}{k}\binom{2k}{k}
&\equiv
-2\sum_{k=1}^{p-1} \frac{v_{k}(2-t)}{k^3} \pmod{p}.
\end{align}
However, because Equations~\eqref{CC3} and~\eqref{CC4} of Theorem~\ref{sec:congruences}
are congruences modulo $p^3$, this simple trick is insufficient to turn them into equivalent congruences
with the central binomial coefficients in the numerators.
To achieve that we need to work a bit harder, as in our next result.

Evaluations of
$\sum_{k=0}^{p-1} \binom{2k}{k}t^{p-1-k}\pmod{p^2}$
and
$\sum_{k=1}^{p-1} \binom{2k}{k}k^{-1}\,t^{p-k}\pmod{p}$
were obtained in~\cite[Equation~(2.2)]{ZWS:10b}
and~\cite[Equation~(1.11)]{ST:10}, respectively,
starting from Equation~\eqref{I7} above, and the further
polynomial identity
\begin{equation}
\label{I8}
\sum_{k=1}^{n-1} \frac{\binom{2k}{k}}{k}\, t^{n-k}=
-2\sum_{d=1}^{n-1}\frac{(-1)^{d}}{d}
\sum_{k=0}^{n-d-1} \binom{2n}{k}v_{n-d-k}(t-2)
-4\sum_{d=1}^{n-1}\frac{(-1)^{d}}{d}\binom{2n-1}{n-d-1},
\end{equation}
which was also proved in~\cite{ST:10}.
The key to push those evaluations in~\cite{ZWS:10b,ST:10} to higher moduli lies in some of the functional equations for the finite polylogarithms
which we have recalled in Section~\ref{sec:polylog}.
The resulting congruences involve the Lucas sequences
\begin{align*}
u_0(x,y)&=0, &u_1(x,y)&=1, &\mbox{and}\quad u_n(x,y)&=x\,u_{n-1}(x,y)-y\,u_ {n-2}(x,y)
\quad\mbox{for $n>1$,}\\
v_0(x,y)&=2, &v_1(x,y)&=x, &\mbox{and}\quad v_n(x,y)&=x\,v_{n-1}(x,y)-y\,v_ {n-2}(x,y)
\quad\mbox{for $n>1$,}
\end{align*}
which generalize the Lucas sequences $u_n(x)=u_n(x,1)$ and $v_n(x)=v_n(x,1)$ introduced in Section~\ref{sec:identities}.
Once again, letting $\alpha$ be an element of a quadratic field extension
of the field $\Q(x,y)$ of rational functions with $\alpha^2-x\alpha+y=0$,
we have
$u_n(x)=(\alpha^n-\alpha^{-n})/(\alpha-\alpha^{-1})$
and
$v_n(x)=\alpha^n+\alpha^{-n}$.

\begin{thm}\label{thm:last}
For any prime $p>3$ we have the polynomial congruences
\begin{align}
\label{CC7}
\sum_{k=0}^{p-1} \binom{2k}{k} t^{p-1-k}
&\equiv
2u_p(t,t)-u_p(2-t)
-2p^2\sum_{k=1}^{p-1}\frac{u_k(2-t)+u_k(t,t)}{k^2} \pmod{p^3},
\\\label{CC8}
\sum_{k=1}^{p-1} \frac{\binom{2k}{k}}{k}\,t^{p-k}
&\equiv
\frac{3t^p+2-v_p(2-t)-4v_p(t,t)}{p} \pmod{p^2}.
\\\label{CC9}
\frac{1}{2}\sum_{k=1}^{p-1} \frac{\binom{2k}{k}}{k^2}\,t^{p-k}
&\equiv
\frac{v_p(2-t)+2v_p(t,t)-t^p-2}{p^2}+\sum_{k=1}^{p-1} \frac{v_{k}(2-t)}{k^2} \pmod{p}.
\end{align}
\end{thm}

Note that if $\alpha$ is an element of a quadratic field extension
of the field $\Q(t)$ of rational functions with $\alpha^2-(2-t)\alpha+1=0$,
whence
$u_n(2-t)=(\alpha^n-\alpha^{-n})/(\alpha-\alpha^{-1})$
and
$v_n(2-t)=\alpha^n+\alpha^{-n}$,
then $1-\alpha$ satisfies
$(1-\alpha)^2-t(1-\alpha)+t=0$,
whence
$u_n(t,t)=\bigl((1-\alpha)^n-(1-\alpha^{-1})^n\bigr)/(\alpha^{-1}-\alpha)$
and
$v_n(t,t)=(1-\alpha)^n+(1-\alpha^{-1})^n$.
Consequently, the sum in the right-hand side of Equation~\eqref{CC7} can be expressed in terms of
$\pounds_2(\alpha^{\pm 1})$
and
$\pounds_2(1-\alpha^{\pm 1})$.
Because the proof of Theorem~\ref{thm:last} is very similar to that of
Theorem~\ref{thm:congruences}, we only outline the argument.

\begin{proof}[Sketch of proof]
The general scheme of proof is to deduce the congruences~\eqref{CC7} and~\eqref{CC8} from the identities~\eqref{I7} and~\eqref{I8}
in a similar way as we deduced the congruences~\eqref{CC1} and~\eqref{CC2} of Theorem~\ref{thm:congruences}
from the identities~\eqref{I3} and~\eqref{I4} of Theorem~\ref{thm:recurrences}.
Thus, after taking $n=p$ and separating one term of the sum we apply standard congruences for binomial coefficients,
but modulo a higher power of $p$ than those needed in the proof of Theorem~\ref{thm:congruences}, such as
\[
\binom{2p}{k}\equiv (-1)^{k-1}\frac{2p}{k}\bigl(1-2pH_{k-1}(1)\bigr)\pmod{p^3},
\qquad\text{for $k=1,\dots,p-1$.}
\]
Because $H_{p-k-1}(1)\equiv H_{k-1}(1)+1/k\pmod{p}$,
the effect of this higher precision is the appearance of new terms, such as
$p^2\sum_{k=1}^{p-1}u(2-t,1)/k^2=
p^2\bigl(\pounds_2(\alpha)-\pounds_2(\alpha^{-1})\bigr)
/(\alpha-\alpha^{-1})$,
in contrast with the proof of Theorem~\ref{thm:congruences},
which only involved $\pounds_1(\alpha)$ and $\pounds_1(\alpha^{-1})$, albeit implicitly.
It is at this place that several congruences from Section~\ref{sec:polylog} for $\pounds_1$ and $\pounds_2$
can be brought into play, at the expense of the appearance of $\pounds_2(1-\alpha)$ and $\pounds_2(1-\alpha^{-1})$
as observed above.

Finally, Equation~\eqref{CC9} can be easily deduced from Equation~\eqref{CC8}
using Equation~\eqref{S1} with $n=p$.
\end{proof}

\section{Numerical congruences}\label{sec:ZWS}

In this final section we illustrate how the special values of the
finite polylogarithms investigated in
Section~\ref{sec:special_values}
allow one to evaluate the polynomial congruences in
Section~\ref{sec:congruences} and~\ref{sec:numerators}
at certain special values of $t$,
thus producing explicit numerical congruences,
some of which can be found proved or conjectured in the literature.

For a given algebraic number $t\not=0$, let $\alpha$ and $\alpha^{-1}$ be the two complex roots of the polynomial $x^2-(2-t)x+1$,
whence $t=2-\alpha-\alpha^{-1}$.
Then for $k\geq 0$ we have
\[
u_k(2-t)=
\begin{cases}
\ds\frac{\alpha^k-\alpha^{-k}}{\alpha-\alpha^{-1}} &\text{if $t\not=4$},\\
(-1)^k k &\text{if $t=4$},
\end{cases}
\quad\text{and}\quad
v_k(2-t)=\alpha^k+\alpha^{-k}.
\]
Consequently, for $d\geq 1$ we have
\[
\sum_{k=1}^{p-1} \frac{u_{k}(2-t)}{k^d}=
\begin{cases}
\ds \frac{\pounds_d(\alpha)-\pounds_d(\alpha^{-1})}{\alpha-\alpha^{-1}} &\text{if $t\not=4$},\\
\pounds_{d-1}(-1) &\text{if $t=4$},
\end{cases}
\quad\mbox{and}\quad
\sum_{k=1}^{p-1} \frac{v_{k}(2-t)}{k^d}={\pounds_d(\alpha)+\pounds_d(\alpha^{-1})}.
\]
Using the special values of $\pounds_d(x)$ established in Section $3$, Theorem~\ref{thm:congruences}
allows one to compute the explicit values of the sums in Equation~\eqref{eq:MC} (modulo $p^3$ or $p$ as stated), for $d=1,2$ and various values of $t$.
Theorem~\ref{thm:identities} then allows one to obtain
analogous formulas for the case $d=0$.

Values of $t$ for which we have quoted or proved congruences for
the corresponding $\pounds_d(\alpha)$
in Section~\ref{sec:special_values} are the following,
grouped together according to $G$-orbits of $\alpha$:
\[
4;\quad 1;\quad -1/2;\quad
2,\,(1\pm i)/2;\quad
3,\,(1\pm i\sqrt{3})/3;\quad
-1,\,2\pm\sqrt{5}=\pm\phi^3.
\]

To illustrate the kind of congruences that one obtains, we give full details of the case $t=-1$, where $\alpha=\phi^2_+$ and $\alpha^{-1}=\phi^2_-$.
In this case $u_n=F_{2n}$ and $v_n=L_{2n}$, where $F_k$ and $L_k$ are respectively the $k$-th Fibonacci number and
the $k$-th Lucas number.
The evaluations modulo $p$ of $\pounds_2(\phi^2_{\pm})$ and $\pounds_3(\phi^2_{\pm})$
which we obtained in Theorem~\ref{thm:phi} yield the following list of congruences.
For comparison, to the right of each congruence we give the sum of the corresponding infinite series,
which can be computed by using Equation~\eqref{eq:asH} and its derivatives at $z=i$.
For reasons of space we omit the moduli from the congruences and specify them in the text.

For any prime $p>5$, Equations~\eqref{CC3},~\eqref{CC4} and~\eqref{I6} yield the following three congruences modulo $p$:
\begin{align*}
\ds p\sum_{k=1}^{p-1}\frac{(-1)^kH_{k-1}(2)}{k\binom{2k}{k}}
&\equiv
\frac{1}{5}\leg{p}{5}q_L^2,
&
\ds\sum_{k=1}^{\infty}\frac{(-1)^kH_{k-1}(2)}{k\binom{2k}{k}}
&=
\frac{4\sqrt{5}}{15}\log^3(\phi_+),
\\
p\ds \sum_{k=1}^{p-1}\frac{(-1)^kH_{k-1}(2)}{k^2\binom{2k}{k}}
&\equiv
\frac{4}{15}\left(\frac{1}{2}q_L^3+B_{p-3}\right),
&
\ds\sum_{k=1}^{\infty}\frac{(-1)^kH_{k-1}(2)}{k^2\binom{2k}{k}}
&=
\frac{2}{3}\log^4(\phi_+),
\\
p\ds \sum_{k=1}^{p-1}\frac{(-1)^kH_{k-1}(2)}{\binom{2k}{k}}
&\equiv
\frac{1}{5}\,q_L+\frac{2}{25}\leg{p}{5}q_L^2,
&
\ds\sum_{k=1}^{\infty}\frac{(-1)^kH_{k-1}(2)}{\binom{2k}{k}}
&=
\frac{2}{5}\log^2(\phi_+)+\frac{8\sqrt{5}}{75}\log^3(\phi_+).
\end{align*}
Equations~\eqref{CC1},~\eqref{CC2} and~\eqref{I5} yield the following three congruences modulo $p^3$:
\[
\begin{array}{lll}
\ds p\sum_{k=1}^{p-1}\frac{(-1)^k}{k\binom{2k}{k}}\equiv \frac{1-L_pF_{p}}{2}+\frac{p^2}{5}\leg{p}{5}q_L^2,
&\ds\sum_{k=1}^{\infty}\frac{(-1)^k}{k\binom{2k}{k}}=-\frac{2\sqrt{5}\log(\phi_+)}{5},\\
\ds p\sum_{k=1}^{p-1}\frac{(-1)^k}{k^2\binom{2k}{k}}\equiv \frac{1-L_p^2}{2p}+\frac{4p^2}{15}\left(\frac{1}{2}q_L^3+B_{p-3}\right),
&\ds\sum_{k=1}^{\infty}\frac{(-1)^k}{k^2\binom{2k}{k}}=-2\log^2(\phi_+),\\
\ds p\sum_{k=1}^{p-1}\frac{(-1)^k}{\binom{2k}{k}}\equiv \frac{p-L_pF_{p}}{5}+\frac{2p^2}{25}\leg{p}{5}q_L^2,
&\ds\sum_{k=1}^{\infty}\frac{(-1)^k}{\binom{2k}{k}}=-\frac{1}{5}-\frac{4\sqrt{5}}{25}\log(\phi_+).
\end{array}
\]

We refrain from listing all of the congruences produced by our results for the remaining values of $t$ listed earlier,
but we point out that, to our knowledge, three of them were already known, and several were conjectured.
The known ones were proved by Z.~W.~Sun in~\cite[Theorems~1.2 and~1.3]{ZWS:09},
namely, \cite[Equation~(1.6)]{ZWS:09} follows from our Equation~\eqref{CC4}  with $t=2$,
while~\cite[Equation~(1.12) and~(1.13)]{ZWS:09} follow from our Equations~\eqref{CC4} and~\eqref{CC5}  with $t=4$.
Furthermore, our congruences confirm some of the conjectures stated by Z.~W.~Sun in~\cite[A31]{ZWS:10}, for $p>3$:
\begin{align*}
&\ds p\sum_{k=1}^{p-1}\frac{2^k}{k\binom{2k}{k}}
\equiv
\leg{-1}{p}-1-pq_p(2)+p^2E_{p-3} \pmod{p^3},
\\
&\ds p\sum_{k=1}^{p-1}\frac{2^k}{k^2\binom{2k}{k}}
\equiv
-q_p(2)+\frac{p^2}{16}B_{p-3}  \pmod{p^3},
\\
&\ds p\sum_{k=1}^{p-1}\frac{4^k}{k^2\binom{2k}{k}}
\equiv
-4\,q_p(2)-2pq_p^2(2)+p^2B_{p-3} \pmod{p^3}.
\end{align*}

Some congruences for the case $d=3$ can be obtained from Theorem~\ref{thm:congruences3}, such as the following,
for $p>3$:
\[
p\sum_{k=1}^{p-1}\frac{4^k}{k^3\binom{2k}{k}}\equiv -4q_p(2)^2+p\left(\frac{4}{3}q_p(2)^3-\frac{1}{6}B_{p-3}\right)\pmod{p^2}.
\]
In this case, values of $t$ different from $4$ are not as easy to deal with.
One further integration using Lemma \ref{lem:integrate} allows one to obtain congruences with $d=4$ as well.
Although we have not stated a corresponding result analogous to Theorem~\ref{thm:congruences3},
we mention that one can derive the congruence, for $p>3$,
\[
p\sum_{k=1}^{p-1}\frac{4^k}{k^4\binom{2k}{k}}\equiv -\frac{4}{3}\left(2q_p(2)^3+B_{p-3}\right)\pmod{p}.
\]

Together with the special values of the finite dilogarithm computed in Section~\ref{sec:special_values},
Equation~\eqref{CC7} of Theorem~\ref{thm:last} allows us to evaluate the sum
$\sum_{k=0}^{p-1} \binom{2k}{k}t^{-k} \pmod{p^3}$
for the values of $t$ mentioned earlier.
Aside from the case $t=4$, which is trivial here because of the identity
$\sum_{k=0}^{n}\binom{2k}{k}4^{-k}
=(2n+1)\binom{2n}{n}4^{-n}$,
two more of these evaluations were already known:
the case $t=-1$ is~\cite[Theorem~1.3]{PS:10}, and
the case $t=2$ is~\cite[Theorem~1.1]{ZWS:10c}.
Our new contributions due to Equation~\eqref{CC7}, for $p>3$, are
\begin{align*}
&\ds \sum_{k=0}^{p-1}{\binom{2k}{k}}
\equiv
\leg{p}{3}-\frac{p^2}{3}\,B_{p-2}(1/3) \pmod{p^3},
\\
&\ds \sum_{k=0}^{p-1}\frac{\binom{2k}{k}}{3^k}
\equiv
\leg{p}{3}-\frac{2p^2}{9}\,B_{p-2}(1/3) \pmod{p^3},
\\
&\ds \sum_{k=1}^{p-1}(-2)^k\binom{2k}{k}
\equiv
-\frac{4p}{3}\,q_p(2)
\pmod{p^3}.
\end{align*}
In a similar way, specializing Equations~\eqref{CC8} and~\eqref{CC9} produces several
numerical congruences.
Among those we mention
\begin{align*}
&\sum_{k=1}^{p-1} \frac{(-1)^k}{k}\binom{2k}{k}
\equiv -2q_L-p\,q_L^2 \pmod{p^2},
\\
&\ds \sum_{k=1}^{p-1}\frac{\binom{2k}{k}}{k^2}
\equiv
\frac{1}{2}\leg{p}{3}\,B_{p-2}\left(\frac{1}{3}\right) \pmod{p},
\end{align*}
which hold for any prime $p>3$.

Even the irrational values of $t$ in our list give rise to nice congruences
with rational terms after combining the algebraic conjugates together.
As an example, because $t=2\mp\sqrt{5}=\phi_{\mp}^3$
corresponds to $\alpha=\pm\phi_{+}$, and because
$2\phi_{\pm}^n=L_n\pm\sqrt{5}F_n$,
Equations~\eqref{CC8} and~\eqref{CC9} yield, for $p>5$,
\begin{align*}
&\ds \sum_{k=1}^{p-1}\binom{2k}{k}\frac{(-1)^k F_{3k-\leg{p}{5}}}{k}
\equiv
\frac{1}{5}\,p\, q_L^2 \pmod{p^2},
\\
&\ds \sum_{k=1}^{p-1}\binom{2k}{k}\frac{(-1)^k L_{3k-\leg{p}{5}}}{k^2}
\equiv
0 \pmod{p}.
\end{align*}


\end{document}